\documentclass[a4paper]{article}

\usepackage{amssymb,amsmath,amsthm,amscd}
\usepackage[english]{babel}

\usepackage{times}
\usepackage[T1]{fontenc}

\usepackage{hyperref}

\newtheorem{proposition}{Proposition}
\newtheorem{theorem}{Theorem}
\newtheorem{corollary}{Corollary}
\newtheorem{remark}{Remark}

\newtheorem{example}{Example}

\newcommand{\HH}{{\mathbb{H}}}

\newcommand{\rr}{{\mathbb{R}}}

\newcommand{\s}{{\mathbb{S}}}

\newcommand{\I}{\mathcal{I}}
\newcommand{\dibar}{\overline\partial}

\newcommand\vs[1]{{#1}_s^\circ}
\newcommand\sd[1]{{#1}'_s}

\newcommand\re{\operatorname{Re}}
\newcommand\im{\operatorname{Im}}

\newcommand{\ui}{\imath}

\newcommand{\OO}{\Omega}

\newcommand{\mbb}{\mathbb}

\newcommand{\R}{\mbb{R}}
\newcommand{\mc}{\mathcal}

\newcommand{\dd[3]}{\frac{\partial #2}{\partial #3}}
 \newcommand{\C}{\mbb{C}}
 
 \newcommand{\Q}{\mc{Q}}
 \def\SS{\mathbb S}
 \newcommand{\bc}{\begin{center}}
 \newcommand{\ec}{\end{center}}

\newcommand{\Zt}{{\mathcal P}}
\newcommand{\Rn}{\mathbb R_n}
\newcommand\Ac{\Rn\otimes_{\R}\C}
\def\Q{\mc{Q}}
\newcommand\IM{\operatorname{Im}}

\title{Almansi-type theorems for slice-regular functions on Clifford algebras}

\author{Alessandro Perotti\\
\small Department of Mathematics, University of Trento\\ 
\small Via Sommarive 14, I-38123 Povo Trento, Italy\\
\small alessandro.perotti@unitn.it}

\date{  }

\begin{document}

\maketitle


\begin{abstract}
We present an Almansi-type decomposition for polynomials with Clifford coefficients, and more generally for slice-regular functions on Clifford algebras. The classical result by Emilio Almansi, published in 1899, dealt with poly\-harmonic functions, the elements of the kernel of an iterated Laplacian. Here we consider polynomials of the form $P(x)=\sum_{k=0}^d x^ka_k$, with Clifford coefficients $a_k\in\mathbb R_{n}$, and get an analogous decomposition related to zonal polyharmonics. We show the relation between such decomposition and the Dirac (or Cauchy-Riemann) operator and extend the results to slice-regular functions.
\medskip



\end{abstract}

\section{Introduction and preliminaries}

A classical result of Emilio Almansi, published in 1899 (see \cite{Almansi,Aronszajn}), gave an expansion of polyharmonic functions in terms of harmonic functions. We recall that for every integer $p\ge1$, a real function $u$ defined on an open set $\OO$ in $\R^N$ is called \emph{polyharmonic} of order $p$ (or \emph{$p$-harmonic}) if $u$ is of class $\mathcal C^{2p}$ and $\Delta^p u=0$ on $\OO$, where $\Delta$ is the Laplacian operator of $\R^N$.

\begin{theorem}[Almansi]\label{th:ClassicalAlmansi}
If $\Delta^p u=0$ on a star-like domain $D\subseteq\rr^N$ with centre 0, then there exist unique harmonic functions $u_0,\ldots,u_{p-1}$ in $D$ such that
\[u(x)=u_0(x)+|x|^2u_1(x)+\cdots+|x|^{2p-2}u_{p-1}(x)\quad\text{for }x\in D.
\]
\end{theorem}

Let $\Rn$ denote  the real Clifford algebra of signature $(0,n)$, with basis vectors $e_1,\ldots, e_n$.
Consider Clifford polynomials of the form
\[p(x)=\sum_{k=0}^dx^ka_k\]
with Clifford coefficients $a_k\in\Rn$ on the right. 
If $n=2m+1$ is odd, these polynomials are polyharmonic of order $m+1$ with respect to the standard Laplacian of $\rr^{n+1}$, i.e.\ $(\Delta_{n+1})^{m+1}p=0$ on $\rr^{n+1}$ (see e.g.\ \cite[Corollary~4.4.4]{Harmonicity} and the next section). Since the Laplacian is a real operator, this property is equivalent to say that the real components of a Clifford polynomial are $(m+1)$-harmonic. 
The classical Almansi's Theorem \ref{th:ClassicalAlmansi} implies then that the real components of a Clifford polynomial have an expansion in terms of a $(m+1)$-tuple of harmonic functions. 

Observe that every Clifford polynomial $p$ is more than $(m+1)$-polyharmonic, since it is in the kernel of the $n^{th}$-order differential operator $\dibar(\Delta_{n+1})^{m}$, where $\dibar$ is the Cauchy-Riemann operator
\[\dibar=\dd{}{x_0}+e_1\dd{}{x_1}+\cdots+e_n\dd{}{x_n}\]
 on $\Rn$ (Theorem~\ref{th:harmonicity}). We recall also the definitions of the Dirac operator
\[\mathcal D=e_1\dd{}{x_1}+\cdots+e_n\dd{}{x_n}\]
and the conjugated Cauchy-Riemann operator 
\[\partial=\dd{}{x_0}-e_1\dd{}{x_1}-\cdots-e_n\dd{}{x_n}\]
on $\Rn$. Since
\[\partial \dibar=\dibar\partial=\Delta_{n+1}=\frac{\partial^2}{\partial x_0^2}+\frac{\partial^2}{\partial x_1^2}+\cdots+\frac{\partial^2}{\partial x_n^2},\]
the operators $\partial,\dibar$ factorize the Laplacian operator of the paravector subspace \[V=\{x_0+x_1e_1+\cdots +x_ne_n\in\Rn\ |\ x_0,\ldots,x_n\in\R\}\simeq\mathbb R^{n+1}\] 
of $\Rn$. This property shows how Clifford analysis, the well-developed function theory based on Dirac and Cauchy-Riemann operators (see \cite{BDS,DSS,GHS} and the vast bibliography therein), is related to the harmonic analysis of the Euclidean space. 

The title of the present paper resembles the one of \cite{MalonekRen}, where the authors proved Almansi-type theorems in Clifford analysis. Note that the two classes of functions, the one of slice-regular functions on $\Rn$ and the one of monogenic functions in the kernel of $\dibar$, are really skew, in the sense that only constant functions are both monogenic and slice-regular (see \cite[Corollary~3.3.3]{Harmonicity}). In particular, standard nonconstant Clifford polynomials of the form $\sum_{k=0}^dx^ka_k$ are not monogenic. However, as also the results obtained in this paper show, the relations between the two function theories deserve interest.

In this paper we will prove, using the symmetry properties of Clifford powers $x^k$, a refined  Almansi-type decomposition for Clifford polynomials and we will extend it to the larger class of slice-regular functions on $\Rn$ and to the case of even dimension $n$.

\subsection{Slice-regular functions on Clifford algebras}

Polynomial functions $f(x)=\sum_k x^ka_k$ with Clifford coefficients are examples of 
\emph{slice-regular} functions. This class of functions constitutes a recent function theory in several hypercomplex settings, including quaternions and real Clifford algebras (see \cite{GeSt2007Adv, CoSaSt2009Israel,GhPe_AIM,libroverde,GeStoSt2013,AlgebraSliceFunctions,DivisionAlgebras}). It was introduced by Gentili and Struppa \cite{GeSt2007Adv} for functions of one quaternionic variable and by Colombo, Sabadini and Struppa \cite{CoSaSt2009Israel} for functions defined on open subsets of $\R^{n+1}$, considered as the paravector space of $\Rn$. 

We briefly recall the basic definitions we need for the main results and refer the reader to the cited references for more details.
Let $\HH$ denote the skew field of quaternions, with basic elements $i,j,k$. For each quaternion $J$ in the sphere of imaginary units
 \[\SS_\HH=\{J\in\HH\ |\ J^2=-1\}=\{x_1i+x_2j+x_3k\in\HH\ |\ x_1^2+x_2^2+x_3^2=1\},\]
let $\C_J=\langle 1,J\rangle\simeq\C$ be the subalgebra generated by $J$. Then we have the ``slice'' decomposition
\[\HH=\bigcup_{J\in \SS_\HH}\C_J, \quad\text{with $\C_J\cap\C_K=\R$\quad for every $J,K\in\SS_\HH,\ J\ne\pm K$.}\]
A differentiable function $f:\OO\subseteq\HH\rightarrow\HH$  is called \emph{(left) slice-regular} \cite{GeSt2007Adv} on the open set $\OO$ if, for each $J\in\SS_\HH$, the restriction
$f_{\,|\OO\cap\C_J}\, : \, \OO\cap\C_J\rightarrow \HH$
is holomorphic with respect to the complex structure defined by left multiplication by $J$. 

If $\OO\subseteq\R^{n+1}$ is an open subset of the paravector space in $\Rn$, a function  $f:\OO\subseteq\R^{n+1}\rightarrow\Rn$ such that, for each $J$ in the sphere $\SS^{n-1}$ of paravector imaginary units, the restriction 
$f_{\,|\OO\cap\C_J}\, : \, \OO\cap\C_J\rightarrow \Rn$
is holomorphic w.r.t.\ the complex structure $L_J(v)=Jv$, is called \emph{slice monogenic} \cite{CoSaSt2009Israel}.
For example, polynomials $f(x)=\sum_k x^ka_k$ with Clifford coefficients on the right are slice monogenic on $\R^{n+1}$. More generally, convergent power series are slice monogenic on an open ball in $\R^{n+1}$ centred at the origin.
Observe that nonconstant polynomials (or power series) do not belong to the kernel of the Cauchy-Riemann operator $\dibar$ of Clifford analysis.

A different approach to slice regularity has been introduced in \cite{GhPe_AIM,AlgebraSliceFunctions} and can be generalized to an ample class of real $^*$- algebras. Here we restrict to the case of the Clifford algebras $\Rn$, in view of their direct relations with Euclidean harmonic analysis and Clifford analysis.
Consider on $\Rn$ the $^*$-algebra structure given by the Clifford conjugation $x\mapsto x^c$, the unique linear antiinvolution of $\Rn$ such that \[e_i^c=-e_i\quad\text{for }i=1,\ldots,n.\]
The real multiples of the unity of $\Rn$ are identified with the real numbers. 
 If $x=x_0+x_1e_1+\cdots +x_ne_n\in\R^{n+1}$, then $x^c=x_0-x_1e_1-\cdots -x_ne_n$ and therefore the \emph{trace}
 $t(x)=x+x^c=2x_0$ is real and the \emph{norm} $n(x)=xx^c=|x|^2$ is real nonnegative. These relations hold on the entire \emph{quadratic cone} of $\Rn$, the subset $\Q_{\Rn}$ of $\Rn$ defined by 
\[\Q_{\Rn}=\{x\in\Rn\;|\; t(x)\in\R,n(x)\in\R\}.\]
Note that $\Q_{\R_1}=\R_{1}\simeq\C$, 
$\Q_{\R_2}=\R_2\simeq\HH$, while $\Q_{\Rn}$ is a proper subset of $\Rn$ if $n\ge3$, properly containing the paravector space.
Each element $x\in\Q_{\Rn}$ can be uniquely written as
\[x=\re(x)+\im(x)\]
with $\re(x)=\frac{x+x^c}2$, $\im(x)=\frac{x-x^c}2=\beta J$, with $\beta=|\im(x)|$ and $J\in\SS_{\Rn}$, the ``sphere'' of all imaginary units $J$ in $\Rn$  (not necessarily paravectors) compatible with the $*$-algebra structure, that is satisfying the conditions $t(J)=0$, $n(J)=1$.
In other words, we still have the slice decomposition
  \[\Q_{\Rn}=\bigcup_{J\in \SS_{\Rn}}\C_J\]
where $\C_J=\langle 1,J\rangle$ is the complex ``slice'' of $\Rn$ generated by $J\in\SS_{\Rn}$. It holds $\C_J\cap\C_K=\R$ for each $J,K\in\SS_{\Rn}$, $J\ne\pm K$. The quadratic cone is a real cone invariant with respect to translations along the real axis.

Let $\Rn\otimes_{\R}\C$ be the complex Clifford algebra, with elements represented as $w=a+\ui b$, $a,b\in\Rn$, $\ui^2=-1$.
Every Clifford polynomial $f(x)=\sum_k x^ka_k$ lifts to a unique polynomial function $F:\C\rightarrow\Rn\otimes_{\R}\C$  which makes the following diagram commutative for all $J \in \SS_{\Rn}$:
\begin{equation}\label{comm_diagram}
\begin{CD}
\C\simeq \R\otimes_\R\C @>F> >\Rn\otimes_\R\C\\ 
@V \Phi_J V  V 
@V V \Phi_J V\\
\Q_{\Rn} @>f> >\Rn 
\end{CD} 
\end{equation}
where $\Phi_J: \Rn\otimes_{\R}\C \to \Rn$ is defined by $\Phi_J(a+\ui b):=a+Jb$. The lifted polynomial is simply $F(z)=\sum_k z^ka_k$, with variable $z=\alpha+\ui \beta\in\C$.
In this lifting, the usual product of polynomials with coefficients in $\Rn$ on one fixed side (the one obtained by imposing commutativity of the indeterminate with the coefficients when two polynomials are multiplied together) corresponds to the pointwise pro\-duct in the algebra $\Rn\otimes_{\R}\C$.

If a Clifford algebra-valued function $f$ (not necessarily a polynomial) has a holomorphic lifting $F$ as in \eqref{comm_diagram} then $f$ is called \emph{(left) slice-regular}.

More precisely, let $D\subseteq\C$ be a set that is invariant with respect to complex conjugation. 
In $\Ac$ consider the complex conjugation that maps $w=a+\ui b$ to $\overline w=a-\ui b$ (with $a,b\in \Rn$).
If a function $F: D \to \Ac$ satisfies  $F(\overline z)=\overline{F(z)}$ for every $z\in D$, then $F$  is called a \emph{stem function} on $D$. Let $\OO_D$ be the \emph{circular} subset of the quadratic cone defined by 
\[\OO_D=\bigcup_{J\in\SS_{\Rn}}\Phi_J(D).\]
 The stem function $F=F_1+\ui F_2:D \to \Ac$  induces the \emph{(left) slice function} $f=\I(F):\OO_D \to \Rn$ in the following way: if $x=\alpha+J\beta =\Phi_J(z)\in \OO_D\cap \C_J$, then  
\[ f(x)=F_1(z)+JF_2(z),\]
where $z=\alpha+\ui\beta$.
The slice function $f=\I(F)$ is called \emph{(left) slice-regular} if $F$ is holomorphic. 
The function $f=\I(F)$ is called \emph{slice-preserving} if $F_1$ and $F_2$ are real-valued (this is the case already considered by Fueter \cite{Fueter1934} for quaternionic functions and by  G\"urlebeck and Spr\"ossig \cite{GHS} for \emph{radially holomorphic} functions on Clifford algebras). In this case, the condition $f(x^c)=f(x)^c$ holds for each $x\in\OO_D$.

When $n=2$, $\R_2$ is the algebra of real quaternions. If the domain $D$ intersects the real axis, this definition of slice regularity is equivalent to the one proposed by Gentili and Struppa \cite{GeSt2007Adv}.
For any $n>2$, when the domain $\OO$ intersects the real axis, the definition of slice regularity on $\Rn$ is equivalent to the one of slice monogenicity, in the sense that the restriction to the paravector space of a $\Rn$-valued slice-regular function is a slice-monogenic function.

\subsection{Slice product and spherical derivatives}

Let $\OO=\OO_D$ be a circular open subset of the quadratic cone of $\Rn$.
The \emph{slice product} of two slice functions $f=\I(F)$, $g=\I(G)$ on $\OO$ is the slice function induced by the pointwise product of the stems $F$ and $G$, i.e.\  $f\cdot g=\I(FG)$. If $f$ is slice-preserving, then $f\cdot g$ coincides with the pointwise product of $f$ and $g$. In this case, we will denote it simply by $fg$. 
Note that if $f,g$ are slice-regular on $\OO$, then also $f\cdot g$ 
 is slice-regular on $\OO$.



The function $\vs f:\OO \to \Rn$, called \emph{spherical value} of $f$, and the function $f'_s:\OO \setminus \R \to \Rn$, called  \emph{spherical derivative} of $f$, are defined as
\[
\vs f(x):=\tfrac{1}{2}(f(x)+f(x^c))
\quad \text{and} \quad
f'_s(x):=\tfrac{1}{2} \im(x)^{-1}(f(x)-f(x^c)).
\] 
 They are slice functions, constant on every set of the form $\SS_x=\alpha+\SS_{\Rn}\beta$ for any $x=\alpha+J\beta\in\OO\setminus\R$. Moreover, they satisfy the relation
  \begin{equation}\label{eq:sp_representation}
  f(x)=\vs{f}(x)+\im(x)f'_s(x)
  \end{equation}
for each $x\in\OO\setminus \R$. 

\begin{remark}
It can be shown that
$\vs{f}(x)$ is the spherical mean of $f$ on $\SS_x$, while $f'_s(x)$ is the spherical mean of $\im(x)^{-1}f(x)$ on $\SS_x$.
\end{remark}

\subsection{Slice-regularity and harmonicity}

In \cite{Harmonicity} some formulas were proved, relating the Cauchy-Riemann operator, the spherical Dirac operator and the spherical derivative of a slice function. From this it was obtained a result which implies, in particular, the Fueter-Sce Theorem and gives a characterization of slice-regularity by means of the Cauchy-Riemann operator.  We recall that Fueter's Theorem \cite{Fueter1934}, generalized by Sce \cite{Sce}, Qian \cite{Qian1997} and Sommen \cite {Sommen2000} on Clifford algebras and octonions, in our language states that applying  to a slice-preserving slice-regular function the Laplacian operator of $\R^4$ (in the quaternionic case) or the iterated Laplacian operator $(\Delta_{n+1})^{m}$ of $\R^{n+1}$ (in the Clifford algebra case with $n=2m+1$ odd), one obtains a function in the kernel, respectively, of the Cauchy-Riemann-Fueter operator or of the Cauchy-Riemann operator.
This result was extended  in \cite{CSSFueter2010} to the whole class of slice-monogenic functions defined by means of stem functions.

Since  the paravector space $\R^{n+1}$ is contained in the quadratic cone $\Q_{\Rn}$, we can consider the restriction of a slice function on domains of the form $\OO=\OO_D\cap\R^{n+1}$ in $\R^{n+1}$. Thanks to the representation formula (see e.g.~\cite[Proposition~6]{GhPe_AIM}), this restriction uniquely determines the slice function. We will therefore use the same symbol to denote the restriction. 

\begin{theorem}\cite[Corollary~3.3.3 and Corollary~3.4.4]{Harmonicity}\label{th:harmonicity}
\\
Let $f:\OO\subseteq\R^{n+1}\to\Rn$ be a  slice function of class $\mathcal{C}^1(\OO)$. Then it holds:
\begin{itemize}
  \item[(a)]
$f$ is slice-regular if and only if $\,\dibar f=(1-n)f'_s$.
\item[(b)]
Let $n=2m+1$ be odd.
 If $f:\OO\subseteq\R^{n+1}\to\R_n$ is (the restriction of) a slice-regular function, then it holds:
\begin{itemize}
\item[(i)]
  $(\Delta_{n+1})^{m}f'_s=0$ on $\OO$, i.e.\ the spherical derivative of a slice-regular function is $m$-harmonic.
\item[(ii)]
 The following generalization of Fueter-Sce Theorem for $\Rn$ holds true:
\[ \dibar(\Delta_{n+1})^{m}f=0\text{\quad on $\OO$}.\]
\item[(iii)]
 $(\Delta_{n+1})^{m+1} f=0$, i.e.\ every slice regular function on $\Rn$ is polyharmonic of order $m+1$.
\end{itemize}
\end{itemize}
\end{theorem}

\begin{remark}\label{rem:ADBW}
In the case of even $n$, the fractional power $(\Delta_{n+1})^m$ of the Laplacian, with $m={\frac{n-1}2}$, can be defined by means of the Fourier transform. Recently, using the results of Qian  \cite{Qian1997}, Altavilla, De Bie and Wutzig \cite[Lemma~4.4]{AltavillaDeBieWutzig} generalized point (i) of part (b) of the previous statement to the case of even dimension $n$, proving that it still holds $(\Delta_{n+1})^{m}f'_s=0$ for every slice-regular function $f$ for which $(\Delta_{n+1})^{m}f$ is a well-defined differentiable function.
\end{remark}

\section{Almansi-type decomposition for  slice-regular functions}

We firstly consider the case of polynomials $p(x)=\sum_{k=0}^dx^ka_k$ with Clifford coefficients in $\Rn$, with $n$ odd. We start with a decomposition of Clifford powers $x^k$ and then extend it by right-linearity.

\begin{proposition}
If $n=2m+1$ is odd, then for every $k\ge0$ and for every $x\in\R^{n+1}\subseteq\R_n$ it holds
\[
x^k=\Zt^m_k(x)-x^c\cdot\Zt^m_{k-1}(x)=\Zt^m_k(x)-x^c\Zt^m_{k-1}(x)
\]
where  for $k\ge0$ the function $\Zt^m_k(x):=\left(x^{k+1}\right)'_s$ is the spherical derivative of the Clifford power $x^{k+1}$ on $\Rn$ and $\Zt^m_{-1}:=0$.
\end{proposition}
\begin{proof}
Applying the Leibniz-type formula 
$(f\cdot g)'_s=f'_s\cdot\vs g+\vs f\cdot g'_s$
 satisfied by the spherical derivative (see \cite[\S5]{GhPe_AIM})  to the power $x^{k+1}$, we  get
\begin{equation}\label{eq:powers}
\sd{(x^{k+1})}=\sd{(x\cdot x^{k})}=\sd{(x)}\vs{(x^{k})}+\vs{(x)}\sd{(x^{k})}=\vs{(x^{k})}+x_0\sd{(x^{k})}.
\end{equation}
Here the slice product coincides with the pointwise product since the functions are slice-preserving.
Using \eqref{eq:sp_representation} and \eqref{eq:powers}, we can write
\[
x^{k}=\vs{(x^k)}+\IM(x)(x^k)'_s=\sd{(x^{k+1})}-x_0\sd{(x^k)}+\IM(x)\vs{(x^k)}=
\sd{(x^{k+1})}-x^c\sd{(x^k)}
\]
which is the formula to be proved.
\end{proof}

For every $k\ge0$, the polynomial $\Zt^m_k(x)$ is a homogeneous polynomial in the variables $x_0,\ldots,x_n$, of (total) degree $k$, with real coefficients (see \cite[\S3.5]{Harmonicity}). From point (b) of Theorem~\ref{th:harmonicity}, it follows that $\Zt^m_k$ is $m$-harmonic. Moreover, being a spherical derivative, it has an axial symmetry with respect to the real axis, i.e.\ $\Zt^m_k\circ T=\Zt^m_k$ for every orthogonal transformation $T$ of $\R^{n+1}$ that fixes the real points. Every polynomial   in the variables $x_0,\ldots x_n$ of the form $\sum_{k=0}^d \Zt^m_k(x) a_k$ has the same properties. It will be called a \emph{zonal $m$-harmonic polynomial with pole 1}.

\begin{example}
Let $n=5$ and $m=2$ and consider the (3-harmonic) powers $x^k$ of the Clifford variable in $\R^6\subseteq\R_5$. It holds, for every $k\ge1$,
\[
x^k=\Zt^2_k(x)-x^c\,\Zt^2_{k-1}(x)\text{\quad  $\forall x\in\R^6\subseteq\R_5$},
\]
with zonal biharmonics $\Zt^2_k(x)$ in $\R^6$. For example, $x^4=\Zt^2_4(x)-x^c\,\Zt^2_{3}(x)$ with
\begin{align*}
\Zt^2_3(x)&=4x_0(x_0^2-x_1^2-x_2^2-x_3^2-x_4^2-x_5^2)\quad\text{and}\\
\Zt^2_4(x)&=5x_0^4-10x_0^2\left(\textstyle\sum_{i=1}^5 x_i^2\right)+\left(\textstyle\sum_{i=1}^5 x_i^2\right)^2.
\end{align*}
It holds $\Delta_6\Zt^2_4=-40x_0^2+8\left(\textstyle\sum_{i=1}^5 x_i^2\right)$ and then $\Delta^2_6\Zt^2_4=0$, as expected. 
\end{example}

\begin{corollary}\label{cor:Almansi}
Let $f(x)=\sum_{k=0}^d x^ka_k$ be a polynomial function of degree $d\ge1$ with Clifford coefficients $a_k\in\Rn$ (with $n=2m+1$ odd). Then there exist two zonal $m$-harmonic polynomials $A$ and $B$ with pole 1, of degrees $d$ and $d-1$ respectively, such that
\[f(x)=A(x)-x^c B(x)\text{\quad for every }x\in\R^{n+1}.
\] 
\end{corollary}

\begin{proof}
In view of the preceding proposition, it suffices to set $A(x):=\sum_{k=0}^d\Zt^m_k(x)a_k$ and $B(x):=\sum_{k=0}^d\Zt^m_{k-1}(x)a_k$. 
\end{proof}

The polynomial $\Zt^m_k(x)$ coincides, up to a real multiplicative constant $c^m_k$, with the unique zonal $m$-harmonic homogeneous polynomial on $\R^{n+1}$ of degree $k$ with pole $1$ (see  \cite[Chapter~5]{HFT} for the harmonic case ($m=1$) and \cite{Render} for the genuine polyharmonic case ($m>1$)). This means that there exists a real constant $c^m_k$ such that
\[
\Zt^m_k(x)=c^m_k\ \mathcal Z^m_k(x,1),
\]
where $\mathcal Z^m_k(x,y)$ is the \emph{reproducing kernel} of the $m$-harmonic homogeneous polynomials of degree $k$ on $\R^{n+1}$ with respect to the Fischer inner product or (with different constants) with respect to the $\mathcal L^2$-product on the unit sphere $\s^n$ of $\R^{n+1}$.
The constants $c^m_k$ can be determined from the property $\Zt^m_k(1)=k+1$. The value $\Zt^m_k(1)$ can be computed using the fact that the spherical derivative $f'_s$ of a slice-regular function $f=\I(F)$ extends to the real points with the values of the slice derivative $\dd {f}{x}=\I\big(\dd {F}{z}\big)$ of $f$ \cite{GhPe_AIM}. Since 
\[\dd{x^{j}}{x}=j x^{j-1}\text{\quad for every }j\ge1,\]
 it holds  $\Zt^m_k(1)=(k+1)(x^k)_{|x=1}=k+1$.

In particular, when $m=1$, $\mathcal Z^1_k(x,y)$ is the (solid) zonal harmonic of degree $k$ and pole $y$ in $\R^4$ (denoted by  $\mathcal Z_k(x,y)$ in \cite[Chapter~5]{HFT}). In this case the constants are $c^1_k=1/(k+1)$ (see \cite[Proposition~3.5.1]{Harmonicity}), i.e.
\[
\Zt^1_k(x)=\frac1{k+1} \mathcal Z_k(x,1).
\]

Now we extend Corollary~\ref{cor:Almansi} in two directions: we prove the Almansi-type decomposition for every slice-regular function of a Clifford variable and we remove the requirement on the dimension $n$ to be odd. 

In the following statement we will consider slice functions that are $m$-harmonic and axially symmetric with respect to the real axis. In accordance with the terms used for polynomials, such functions will be called \emph{zonal $m$-harmonic functions with pole 1}. For brevity, we use the notation $\partial_\alpha$ and $\partial_\beta$ to denote the partial derivatives with respect to $\alpha$ and $\beta$.

\begin{theorem}\label{th:Almansi}
Let $n\ge2$ be any integer and let $m={\frac{n-1}2}$. Let $f:\OO\subseteq\R^{n+1}\to\Rn$  be slice-regular. 
If $n$ is even, we assume that the fractional Laplacian $(\Delta_{n+1})^{m}f$ is a well-defined differentiable function.
Then there exist two unique $\Rn$-valued zonal $m$-harmonic functions $A$, $B$ with pole $1$ on $\OO$, such that 
\[
f(x)=A(x)- x^c B(x)\text{\quad  $\forall x\in\OO$.}
\]
Conversely, if $A$ and $B$ are  $\Rn$-valued functions of class $C^1$ 
on $\OO$,  axially symmetric with respect to the real axis, 
then $g(x):=A(x)-x^c B(x)$ is a slice function on $\OO$. The function $g$ is slice-regular if and only if $A$ and $B$ satisfy the system of equations
\begin{equation}
\begin{cases}\label{system}
\partial_\alpha A-\alpha\,\partial_\alpha B-\beta\,\partial_\beta B=2B\\
\partial_\beta A-\alpha\,\partial_\beta B+\beta\,\partial_\alpha B=0 
\end{cases}
\end{equation}
where $\alpha=x_0$ and $\beta=|\im(x)|=\sqrt{x_1^2+\cdots+x_n^2}$.
\end{theorem}

\begin{proof}
Assume that $f$ is slice-regular on $\OO$. Then also $x\cdot f=xf$ is slice-regular and it holds
\begin{equation}\label{eq:product}
\sd{(xf)}=\sd{(x\cdot f)}=\sd{(x)}\vs{f}+\vs{(x)}\sd{f}=\vs{f}+x_0\sd{f}.
\end{equation}
From \eqref{eq:sp_representation} and \eqref{eq:product} we get
\[f(x)=\vs{f}(x)+\im(x)f'_s(x)=\left(\vs{f}(x)+x_0 f'_s(x)\right)- x^c\, f'_s(x)=\sd{(xf)}- x^c\, f'_s(x).
\]
Define $A(x):=\vs{f}(x)+x_0 f'_s(x)=(xf)'_s$ and $B(x):=f'_s(x)$.  Point (b) of Theorem~\ref{th:harmonicity}, together with Remark~\ref{rem:ADBW}, implies that  $A$ and $B$, being spherical derivatives of slice regular functions, are zonal $m$-harmonic slice functions on $\OO$ with pole 1.

To prove uniqueness, we observe that if $f(x)=A(x)- x^c B(x)$ with $A$ and $B$ axially symmetric functions w.r.t.\ the real axis, then for every $x\in\OO\setminus\R$ it holds
\[
\sd f(x)=(2\IM(x))^{-1}\left(A(x)-x^c B(x)-A(x^c)+xB(x^c)\right).
\]
Since $A(x^c)=A(x)$ and $B(x^c)=B(x)$, we get that $\sd f(x)=(2\IM(x))^{-1}(x-x^c)B(x)=B(x)$. From this we deduce also $A(x)=f(x)+x^c\sd f(x)$ and therefore $A$ and $B$ are uniquely determined by $f$.

Conversely, assume that $A$ and $B$ are axially symmetric w.r.t.\ the real axis on $\OO=\OO_D$. Given $x=\alpha+J\beta\in\OO$ and $z=\alpha+\ui\beta\in D$, we set $G_1(z):=A(x)-x_0 B(x)=A(x)-\alpha B(x)$ and $G_2(z):=\beta B(x)$. Thanks to the symmetry properties of $A$ and $B$, $G_1$ and $G_2$ are well-defined functions from $D$ into $\Rn$, and the function $G:=G_1+\ui G_2$ is a stem function of class $C^1$ on $D$, which induces the slice function $g=A-x^cB$:
\[\I(G)(x)=A(x)-x_0 B(x)+J\beta B(x)=A(x)-(x_0-\IM(x))B(x)=A(x)-x^cB(x).
\]
To conclude the proof it is sufficient to see that the Cauchy-Riemann equations $\partial_\alpha G_1=\partial_\beta G_2$, $\partial_\beta G_1=-\partial_\alpha G_2$
for $G$ are equivalent to the system \eqref{system} for $A$ and $B$. 
\end{proof}

Note that the correspondence $f\mapsto (A,B)$ given by Theorem~\ref{th:Almansi} is right $\Rn$-linear. Moreover, the pairs $(A,B)$ formed by real-valued functions $A$ and $B$ correspond to the subclass of slice-preserving functions $f$.

\begin{example}
For $n=3$, $m=1$, the functions $A,B$ associated to $f$ are $\R_3$-valued harmonic function of four real variables.
For example, let 
\[f(x)=\exp(x)=e^{x_0}\cos\sqrt{x_1^2+x_2^2+x_3^2}+\frac{\IM(x)}{|\IM(x)|}e^{x_0}\sin\sqrt{x_1^2+x_2^2+x_3^2}\]
 be the slice-preserving regular function induced on $\Q_{\R_3}$ by the complex exponential. Then $f(x)=A(x)-x^c B(x)$ with harmonic components $A,B$ on $\R^4$ given by 
\[
\begin{cases}
A(x)=e^{x_0}\big(\cos\sqrt{x_1^2+x_2^2+x_3^2}+x_0\tfrac{\sin\sqrt{x_1^2+x_2^2+x_3^2}}{\sqrt{x_1^2+x_2^2+x_3^2}}\big),\\
B(x)=e^{x_0}\tfrac{\sin\sqrt{x_1^2+x_2^2+x_3^2}}{\sqrt{x_1^2+x_2^2+x_3^2}}.
\end{cases}
\]
\end{example}

\begin{example}
For $n=2$, $m=1/2$, the function $g(x)=x^{-1}=x^c|x|^{-2}$ is slice-regular on $\R^3\setminus\{0\}\subset\R_2$.
The function $\Delta_3^{1/2}g$ is a multiple of the monogenic Cauchy kernel $E(x)=x^c|x|^{-3}$ on $\R_2$.
It holds
\[\sd g=-\dibar g=-|x|^{-2}\quad\text{and}\quad \sd{(xg)}=0.
\]
Therefore the decomposition of Theorem~\ref{th:Almansi} is given by $A\equiv0$, $B(x)=-|x|^{-2}$. These functions are in the kernel of the fractional Laplacian $\Delta_3^{1/2}$ on $\R^3\setminus\{0\}$.
\end{example}

\begin{remark}
The $m$-harmonic functions $A$ and $B$ in the statement of Theorem \ref{th:Almansi} can also be computed from $f$ by differentiation using the Cauchy-Riemann operator on $\Rn$. Since $A(x)=(xf)'_s$ and $B(x)=f'_s(x)$, from point (a) of Theorem~\ref{th:harmonicity} we get
\[
A=\tfrac1{1-n}\dibar (xf)\quad\text{and}\quad  B=\tfrac1{1-n}\dibar f.
\]
\end{remark}

Combining Theorem~\ref{th:Almansi} with the classical Almansi's Theorem, we obtain another decomposition in terms of biharmonic functions in the kernel of the operator $\dibar\Delta_{n+1}$. Here we assume $m>1$ because the case $m=1$ is already contained in Theorem~\ref{th:harmonicity}.

\begin{corollary}
Let $n=2m+1>3$ be odd. Let $\OO\subseteq\R^{n+1}$ be a star-like domain with centre 0. Let $f:\OO\subseteq\R^{n+1}\to\Rn$  be slice-regular. Then there exist $\Rn$-valued zonal biharmonic functions $g_0,\ldots,g_{m-1}$  on $\OO$, with pole $1$, such that 
\[
f(x)=g_0(x)+|x|^2g_1(x)+\cdots+|x|^{2m-2}g_{m-1}(x)\text{\quad  $\forall x\in\OO$.}
\]
The functions $g_0,\ldots,g_{m-1}$ belong to the kernel of the operator $\dibar\Delta_{n+1}$.
\end{corollary}

\begin{proof}
Let $f=A-x^cB$ be the decomposition given in Theorem~\ref{th:Almansi}. Applying Theorem~\ref{th:ClassicalAlmansi} to the $2^n$ real components of $A$ and $B$, which are $m$-harmonic on $\OO$, we get $\Rn$-valued harmonic functions $u_0,\ldots,u_{m-1}$ and $v_0,\ldots,v_{m-1}$ such that 
\[
A(x)=\sum_{k=0}^{m-1}|x|^{2k}u_k(x)\quad\text{and}\quad B(x)=\sum_{k=0}^{m-1}|x|^{2k}v_k(x).
\]
Therefore
\[
f(x)=A(x)-x^cB(x)=\sum_{k=0}^{m-1}|x|^{2k}(u_k(x)-x^c v_k(x))=\sum_{k=0}^{m-1}|x|^{2k}g_k(x).
\]
The functions $g_k=u_k-x^c v_k$ have the same symmetry properties as the functions $A$ and $B$. This follows from the uniqueness of the Almansi decomposition. Given any orthogonal transformation $T$ of $\R^{n+1}$ that fixes the real points, also $A\circ T$ is $m$-harmonic and $u_k\circ T$ is harmonic. Since
\[
A(T(x))=\sum_{k=0}^{m-1}|x|^{2k}u_k(T(x))=A(x),
\] 
it must be $u_k(T(x))=u_k(x)$ for every $k$. The same holds for $B$. To prove the last statement of the thesis, we observe that for every $\Rn$-valued function $u$, it holds $\Delta_{n+1}(x^c u)=2\partial u+x^c\Delta_{n+1}u$. Therefore 
$\Delta_{n+1}g_k=\Delta_{n+1}(-x^c v_k)=-2\partial v_k$ and then $\dibar\Delta_{n+1}g_k=-2\dibar\partial v_k=-2\Delta_{n+1} v_k=0$. In particular, the functions $g_k$ are biharmonic.
\end{proof}

\section{Acknowledgements}
This work was supported by GNSAGA of INdAM, and by the grants ``Progetto di Ricerca INdAM, Teoria delle funzioni ipercomplesse e applicazioni'', and PRIN ``Real and Complex Manifolds: Topology, Geometry and holomorphic dynamics'' of the Italian Ministry of Education.


\begin{thebibliography}{10}

\bibitem{Almansi}
E.~Almansi.
\newblock {Sull'integrazione dell'equazione differenziale
  \(\varDelta^{2n}=0\)}.
\newblock {\em {Annali di Mat. (3)}}, 2:1--51, 1899.

\bibitem{AltavillaDeBieWutzig}
A.~Altavilla, H.~De~Bie, and M.~Wutzig.
\newblock Implementing zonal harmonics with the {F}ueter principle.
\newblock 2019.
\newblock arXiv:1903.08914.

\bibitem{Aronszajn}
N.~Aronszajn, T.~M. Creese, and L.~J. Lipkin.
\newblock {\em Polyharmonic functions}.
\newblock Oxford Mathematical Monographs. The Clarendon Press, Oxford
  University Press, New York, 1983.

\bibitem{HFT}
S.~Axler, P.~Bourdon, and W.~Ramey.
\newblock {\em Harmonic function theory}, volume 137 of {\em Graduate Texts in
  Mathematics}.
\newblock Springer-Verlag, New York, second edition, 2001.

\bibitem{BDS}
F.~Brackx, R.~Delanghe, and F.~Sommen.
\newblock {\em Clifford analysis}, volume~76 of {\em Research Notes in
  Mathematics}.
\newblock Pitman (Advanced Publishing Program), Boston, MA, 1982.

\bibitem{CSSFueter2010}
F.~Colombo, I.~Sabadini, and F.~Sommen.
\newblock The {F}ueter mapping theorem in integral form and the {$\mathcal
  F$}-functional calculus.
\newblock {\em Math. Methods Appl. Sci.}, 33(17):2050--2066, 2010.

\bibitem{CoSaSt2009Israel}
F.~Colombo, I.~Sabadini, and D.~C. Struppa.
\newblock Slice monogenic functions.
\newblock {\em Israel J. Math.}, 171:385--403, 2009.

\bibitem{libroverde}
F.~Colombo, I.~Sabadini, and D.~C. Struppa.
\newblock {\em Noncommutative functional calculus}, volume 289 of {\em Progress
  in Mathematics}.
\newblock Birkh\"auser/Springer Basel AG, Basel, 2011.
\newblock Theory and applications of slice hyperholomorphic functions.

\bibitem{DSS}
R.~Delanghe, F.~Sommen, and V.~Sou\v{c}ek.
\newblock {\em Clifford algebra and spinor-valued functions}, volume~53 of {\em
  Mathematics and its Applications}.
\newblock Kluwer Academic Publishers Group, Dordrecht, 1992.

\bibitem{Fueter1934}
R.~Fueter.
\newblock Die {F}unktionentheorie der {D}ifferentialgleichungen {$\Delta u=0$}
  und {$\Delta\Delta u=0$} mit vier reellen {V}ariablen.
\newblock {\em Comment. Math. Helv.}, 7(1):307--330, 1934.

\bibitem{GeStoSt2013}
G.~Gentili, C.~Stoppato, and D.~C. Struppa.
\newblock {\em Regular Functions of a Quaternionic Variable}.
\newblock Springer Monographs in Mathematics. Springer, 2013.

\bibitem{GeSt2007Adv}
G.~Gentili and D.~C. Struppa.
\newblock A new theory of regular functions of a quaternionic variable.
\newblock {\em Adv. Math.}, 216(1):279--301, 2007.

\bibitem{GhPe_AIM}
R.~Ghiloni and A.~Perotti.
\newblock Slice regular functions on real alternative algebras.
\newblock {\em Adv. Math.}, 226(2):1662--1691, 2011.

\bibitem{AlgebraSliceFunctions}
R.~Ghiloni, A.~Perotti, and C.~Stoppato.
\newblock The algebra of slice functions.
\newblock {\em Trans. Amer. Math. Soc.}, 369(7):4725--4762, 2017.

\bibitem{DivisionAlgebras}
R.~Ghiloni, A.~Perotti, and C.~Stoppato.
\newblock {D}ivision algebras of slice functions.
\newblock {\em Proceedings A of the Royal Society of Edinburgh: Section A
  Mathematics}, 2019.
\newblock DOI:10.1017/prm.2019.13.

\bibitem{GHS}
K.~G{\"u}rlebeck, K.~Habetha, and W.~Spr{\"o}{\ss}ig.
\newblock {\em Holomorphic functions in the plane and {$n$}-dimensional space}.
\newblock Birkh\"auser Verlag, Basel, 2008.

\bibitem{MalonekRen}
H.~R. Malonek and G.~Ren.
\newblock Almansi-type theorems in {C}lifford analysis.
\newblock {\em Math. Methods Appl. Sci.}, 25(16-18):1541--1552, 2002.

\bibitem{Harmonicity}
A.~Perotti.
\newblock {S}lice regularity and harmonicity on {C}lifford algebras.
\newblock In {\em {T}opics in {C}lifford {A}nalysis -- {S}pecial {V}olume in
  {H}onor of {W}olfgang {S}pr\"o\ss ig}, Trends Math. Springer, Basel, 2019.
\newblock DOI: 10.1007/978-3-030-23854-4.

\bibitem{Qian1997}
T.~Qian.
\newblock Generalization of {F}ueter's result to {${\bf R}\sp {n+1}$}.
\newblock {\em Atti Accad. Naz. Lincei Cl. Sci. Fis. Mat. Natur. Rend. Lincei
  (9) Mat. Appl.}, 8(2):111--117, 1997.

\bibitem{Render}
H.~Render.
\newblock Reproducing kernels for polyharmonic polynomials.
\newblock {\em Arch. Math. (Basel)}, 91(2):136--144, 2008.

\bibitem{Sce}
M.~Sce.
\newblock Osservazioni sulle serie di potenze nei moduli quadratici.
\newblock {\em Atti Accad. Naz. Lincei. Rend. Cl. Sci. Fis. Mat. Nat. (8)},
  23:220--225, 1957.

\bibitem{Sommen2000}
F.~Sommen.
\newblock On a generalization of {F}ueter's theorem.
\newblock {\em Z. Anal. Anwendungen}, 19(4):899--902, 2000.

\end{thebibliography}

\end{document}